\documentclass[12pt]{amsart}

\usepackage{color}
\usepackage{amsthm}
\usepackage{amsmath}
\usepackage{amssymb}
\usepackage{amstext}
\usepackage[left=3.5cm, right=3.5cm, paperheight=11.4in]{geometry}
\usepackage{comment}
\usepackage{enumitem}
\usepackage{ulem}

\usepackage{fancyhdr}
\newtheorem{theorem}{Theorem}
\newtheorem{lemma}[theorem]{Lemma}
\newtheorem{proposition}[theorem]{Proposition}
\newtheorem{corollary}[theorem]{Corollary}

\theoremstyle{definition}
\newtheorem{definition}[theorem]{Definition}

\newtheorem{remark}[theorem]{Remark}

\newcommand\ord{{\rm ord}}

\begin{document}
\title{A Cauchy-Davenport theorem for semigroups}
\author{Salvatore Tringali}
\address{Laboratoire Jacques-Louis Lions,
Universit\'e Pierre et Marie Curie, 4 place Jussieu, 75005 Paris cedex 05, France.}
\email{tringali@ann.jussieu.fr.}
%
%
\thanks{This research was funded from the
European Community's 7th Framework Programme (FP7/2007-2013) under Grant
Agreement No. 276487 (project ApProCEM), and partly from the ANR Project No. ANR-12-BS01-0011 (project Caesar).}
\begin{abstract}
We generalize the Davenport transform and use it to prove that, for
a (possibly non-commutative) cancellative semigroup $\mathbb A = (A, +)$ and non-empty
subsets $X,Y$ of $A$ such that the sub\-sem\-i\-group generated by $Y$ is commutative, we have $|X + Y| \ge \min(\omega(Y), |X| + |Y| - 1)$, where
\begin{displaymath}
\omega(Y) := \sup_{y_0 \in Y \cap \mathbb A^{\times}}
\inf_{y \in Y \setminus \{y_0\}} |\langle y - y_0 \rangle|.
\end{displaymath}
This carries over the Cauchy-Davenport theorem to the broader setting of semigroups, and it implies, in particular, an extension of I. Chowla's and S.S. Pillai's theorems for cyclic groups and a notable strengthening of another generalization of the same Cauchy-Davenport theorem to commutative groups, where $\omega(Y)$ in the above is replaced by the minimal order of the non-trivial subgroups of $\mathbb A$.
\end{abstract}

\maketitle

\section{Introduction}\label{sec:intro}
The present paper deals with the structure theory of \textit{semigroups}. We refer to \cite{BourSetTh}, \cite[Chapter I, Sections 1-2, 4, and 6]{BourAlgI}, and \cite[Chapter 1]{Howie96} for all necessary prerequisites as well as for notation and terminology used but not defined here.

Semigroups are a natural framework for developing large parts of
theories traditionally presented in much less general contexts.
Not only this can suggest new directions of research and shed light on questions primarily focused on groups, but it also makes methods and results otherwise restricted to ``richer
settings'' applicable to significantly larger
classes of problems.

Here, a semigroup is a pair $\mathbb A = (A, +)$ consisting of a (possibly empty) set $A$, referred to as the carrier of $\mathbb A$, and an associative binary operation $+$ on $A$
(unless otherwise specified, all semigroups considered below are written
additively, but are not necessarily commutative).

Given subsets $X, Y$
of $A$, we define as usual the sumset, relative to $\mathbb A$, of the pair $(X,Y)$ as the set $
 X + Y := \{x + y: x \in X, y \in Y\}$,
which is written as $x + Y$ (respectively, $X + y$) if $X = \{x\}$ (respectively, $Y = \{y\}$). Furthermore, we extend the notion of difference set from groups to semigroups by
\begin{displaymath}
X - Y := \{z \in A: (z+Y) \cap X \ne \emptyset\},\quad -X + Y := \{z \in A: (X + z) \cap Y \ne \emptyset\}.
\end{displaymath}
Expressions of the form $Z_1 + \cdots + Z_n$ or $\sum_{i=1}^n Z_i$, involving one or more summands, as well as $-x+Y$ and $X-y$ for $x,y \in A$ are defined in a similar way; in particular, we use $nZ$ for $Z_1 + \cdots + Z_n$ if $Z_1 = \cdots = Z_n =: Z$.

We say that $\mathbb A$ is unital, or a \textit{monoid}, if there exists $0 \in A$ such that $z + 0 = 0 + z = z$ for all $z$; when this is the case, $0$ is unique and called \textit{the} identity of $\mathbb A$.
Then, we let $\mathbb A^\times$ be the set of units of $\mathbb A$, in such a way that $\mathbb A^\times := \emptyset$ if $\mathbb A$ is not
a monoid; this is simply denoted as $A^\times$ if there is no likelihood of confusion. If $\mathbb A$ is unital with identity $0$, a \textit{unit} of $\mathbb A$
is now an element $z$ for which there exists $\tilde z$, provably unique and called
\textit{the} inverse of $z$ in $\mathbb A$, such that $z + \tilde z =
\tilde z + z = 0$. Moreover, for $Z \subseteq A$ we write
$\langle Z \rangle_\mathbb{A}$ for the smallest subsemigroup of $\mathbb A$
containing $Z$, and given $z \in A$ we use $\langle z \rangle_\mathbb{A}$ for $\langle \{z\} \rangle_\mathbb{A}$ and $\ord_\mathbb{A}(z)$
for the order of $z$ in $\mathbb A$, that is $\ord_\mathbb{A}(z)
:= |\langle z \rangle_\mathbb{A}|$, so generalizing the notion of order for the elements of a group. Here and later, the subscript `$\mathbb A$' may be omitted if $\mathbb A$ is clear from the context. Finally, we say that $\mathbb A$ is cancellative if for $x,y,z \in A$ it holds $z+x=z+y$ or $x+z=y+z$ only if $x = y$; we notice that any group is a cancellative monoid.

Sumsets in (mostly commutative) groups have been intensively investigated
for several years (see \cite{Ruzsa09} for a recent survey), and
interesting results have been also obtained in the case of commutative cancellative monoids
\cite{Gerold09} (in A. Geroldinger's work these are simply termed monoids). The present paper aims to extend aspects of
the theory to the more general setting of possibly \textit{non-commutative} semigroups.

Historically, the first significant achievement in the field is
probably the Cauchy-Davenport theorem, originally established by A.-L.~Cauchy
\cite{Cauchy1813} in 1813, and independently rediscovered by H.~Davenport
\cite{Daven35,Daven47} more than a century later:
\begin{theorem}[Cauchy-Davenport theorem]
\label{th:basic_CD}
Let $(A,+)$ be a group of prime order $p$
and $X,Y$ non-empty subsets of $A$. Then, $|X+Y| \ge \min(p,
|X|+|Y|-1)$.
\end{theorem}
The result has been the subject of numerous papers, and
received many different proofs, each favoring alternative points of view
and eventually leading to  progress on analogous questions. In fact, the main contribution here is an extension of Theorem \ref{th:basic_CD} to cancellative semigroups (this is stated in Section \ref{sec:statements}).

The Cauchy-Davenport theorem applies especially to the
additive group of the integers modulo a prime.
Extensions to composite moduli have been given by several authors, and notably by I. Chowla \cite{Chowla35} and S.S. Pillai \cite{Pillai}. These results, reported below for the sake of exposition and used by Chowla and Pillai in relation to Waring's problem, are further strengthened, in Section \ref{sec:statements}, by Corollary \ref{th:chowla_plus}, which can be viewed as a common generalization of both of them, and whose proof is sensibly shorter than each of the proofs appearing in \cite{Chowla35} and \cite{Pillai} (not to mention that it comes as a by-product of a deeper result). Here and later, for $m \in \mathbb N^+$ we write, as usual, $\mathbb Z/m\mathbb Z$ for the integers modulo $m$, endowed with their usual additive and multiplicative structure.

\begin{theorem}[Chowla's theorem]
\label{th:chowla_theorem}
Let $m$ be an integer $\ge 1$. If $X,Y$ are non-empty subsets of $\mathbb Z/m\mathbb Z$ such that $0
\in Y$ and $\gcd(m,y) = 1$ for each $y \in Y \setminus \{0\}$, then $|X+Y|
\ge \min(m, |X|+|Y|-1)$.
\end{theorem}

\begin{theorem}[Pillai's theorem]
\label{th:pillai_theorem}
Given an integer $m \ge 1$, pick non-empty subsets $X,Y$ of $\mathbb Z/m\mathbb Z$. Let $\delta$ be the maximum of $\gcd(m, y - y_0)$ for distinct $y,y_0 \in Y$
if $|Y| \ge 2$, and set $\delta := 1$ otherwise. Then, $|X+Y|
\ge \min(\delta^{-1} m, |X|+|Y|-1)$.
\end{theorem}
A partial account of further results in the same spirit can be found in
\cite[Section 2.3]{Natha}, along with an entire chapter dedicated to
Kneser's theorem \cite[Chapter 4]{Natha}, which, among the other
things, implies Theorem \ref{th:chowla_theorem} (and then also Theorem \ref{th:basic_CD}); see \cite[Section 4.6, Exercises 5 and 6]{Natha}.
Generalizations of the Cauchy-Davenport theorem of a somewhat different flavor have been furnished, still in recent
years, by several authors.

For, assume for the rest of the paper that $\mathbb A = (A, +)$ is a fixed, arbitrary semigroup (unless differently specified), and let $0$ be the identity of the unitization, $\mathbb A^{(1)}$, of $\mathbb A$: If $\mathbb A$ is \textit{not} unital, $\mathbb A^{(1)}$ is the pair $(A \cup \{A\}, +)$, where $+$ is, by an abuse of notation, the unique extension of $+$ to a binary operation on $A \cup \{A\}$ for which $A$ serves as an identity (note that $A \notin A$, so loosely speaking we are just adjoining a distinguished element to $A$ and extending the structure of $\mathbb A$ in such a way that the outcome is a monoid whose identity is the adjoined element); otherwise $\mathbb A^{(1)} := \mathbb A$ (cf. \cite[p. 2]{Howie96}). We denote by $\mathfrak{p}(\mathbb A)$ the minimum of $\ord_{\mathbb A^{(1)}}(z)$ as $z$ ranges in the carrier of $\mathbb A^{(1)}$ and $z \ne 0$, with the convention that $\mathfrak{p}(\mathbb A):= |\mathbb N|$ if $\mathbb A^{(1)} = \{0\}$, namely
$\mathbb A^{(1)}$ is trivial. Then we have:
\begin{theorem}[folklore]
\label{th:karolyi_abelian}
If $\mathbb A$ is a commutative group and $X,Y$ are non-empty subsets of
$A$, then $|X+Y| \ge \min(\mathfrak{p}(\mathbb A), |X| + |Y| - 1)$.
\end{theorem}
Theorem \ref{th:karolyi_abelian} is another (straightforward) consequence of Kneser's theorem. While it applies to both finite and infinite
\textit{commutative}
groups, an analogous result holds true for all groups:
\begin{theorem}[Hamidoune-K\'arolyi theorem]
\label{th:karolyi_finite}
If $\mathbb A$ is a group and $X,Y$ are non-empty subsets of
$A$, then $|X+Y| \ge \min(\mathfrak{p}(\mathbb A), |X| + |Y|-1)$.
\end{theorem}
This was first proved by K\'arolyi in the case of finite groups, relying on the structure theory of
group extensions, by reduction to finite solvable groups in the light
of the Feit-Thompson theorem, and then by Hamidoune in the general case, based on the isoperimetric method; see \cite{Karo05.2} for details.

A further result from the literature that is significant in relation to the subject matter is due to J.H.B.~Kemperman \cite{Kemp56}, and reads as follows:
\begin{theorem}[Kemperman's inequality for torsion-free groups]
\label{th:CD_for_torsion-free_groups}
 Let $\mathbb A$ be a group, and let $X,Y$ be non-empty subsets of
$A$. Suppose that every non-zero element of $A$ has order $\ge
|X|+|Y|-1$ in $\mathbb A$. Then, $|X+Y| \ge |X| + |Y| - 1$.
\end{theorem}
In fact, \cite{Kemp56} is focused on cancellative
semigroups (there simply called semigroups), and it is precisely in this
framework that Kemperman establishes a series of results,
mostly related to the number of different representations of an element in a
sumset, eventually leading to Theorem
\ref{th:CD_for_torsion-free_groups}, a weak version of which will be proved in Section \ref{sec:main_results} as a corollary of our main theorem (namely, Corollary \ref{cor:weak_kemperman}).

For the rest, Hamidoune and coauthors, see \cite[Theorem
3]{Cill10}, have proved a Cauchy-Davenport theorem for \textit{acyclic}
monoids (these are termed acyclic \textit{semigroups} in \cite{Cill10}, but they are, in fact, \textit{monoids} in our terminology), and it would be interesting to find a common pattern among their result and the ones in the present paper; unluckily, the author has no clue on this for the moment (in particular, note that acyclic semigroups in \cite{Cill10} are not cancellative semigroups).
\subsection*{Organization.}
In Section
\ref{sec:statements}, we define the Cauchy-Davenport constant of a pair of sets in a semigroup and state our main results. In Section
\ref{sec:preliminaries}, we establish a few basic lemmas. Section
\ref{sec:davenport_transform} is devoted to generalized Davenport transforms
and their fundamental properties. We demonstrate the central theorem of the paper (namely, Theorem \ref{th:strong_main}) in Section
\ref{sec:main_results} and give a couple of applications in Section
\ref{sec:corollaries}.
\section{The statement of the main results}
\label{sec:statements}

With all the above in mind, we can now proceed to the heart of the paper.
\begin{definition}\label{def:davenport_constant}
For a subset $Z$ of
$A$, we let
\begin{equation}\label{equ:davenport_constant}
\omega_\mathbb{A}(Z) := \sup_{z_0 \in Z \cap A^\times}
\inf_{z \in Z \setminus \{z_0\}} \ord(z - z_0).
\end{equation}
Then, given $X,Y \subseteq A$ we define $\Omega_\mathbb{A}(X,Y) :=
0$ if either of $X$ or $Y$ is empty; $\Omega_\mathbb{A}(X,Y) := \max(|X|,|Y|)$ if $X \times Y \ne \emptyset$ and either
$X$ or $Y$ is infinite, and
\begin{displaymath}
\Omega_\mathbb{A}(X,Y) :=
\min(\omega_\mathbb{A}(X,Y), |X|+|Y|-1)
\end{displaymath}
otherwise, where $\omega_\mathbb{A}(X,Y) := \max(\omega_\mathbb{A}(X),\omega_\mathbb{A}(Y))$. We refer to
$\Omega_\mathbb{A}(X,Y)$ as the \textit{Cauchy-Davenport constant of
$(X,Y)$ relative to $\mathbb A$} (again, the subscript `$\mathbb A$' may be omitted from the notation if there is no danger of ambiguity).
\end{definition}
Here and later, we assume that the
supremum of the empty set is $0$, while its infimum is $|\mathbb N|$, so any pair of subsets of $A$ has a well-defined
Cauchy-Davenport constant (relative to $\mathbb A$). In particular, $\omega(Z)$ is zero for $Z \subseteq A$ if $Z \cap A^\times =
\emptyset$. However, this is not the case, for instance, when $Z \ne \emptyset$ and $\mathbb A$ is a group, which is the ``base'' for the following non-trivial bound:
\begin{theorem}\label{th:strong_main}
Suppose $\mathbb A$ is cancellative and let $X,Y$ be subsets of $A$ such that $\langle Y \rangle$ is commutative. Then, $|X+Y| \ge
\min(\omega(Y),|X|+|Y|-1)$ if both of $X$ and $Y$ are finite and non-empty, and $|X+Y| \ge \Omega(X,Y)$ otherwise.
\end{theorem}
Theorem \ref{th:strong_main} represents the central contribution of the
paper.
Not only it extends the Cauchy-Davenport theorem to the broader and more
abstract setting of semigroups (see Section \ref{sec:corollaries}), but it also
provides a strengthening and a generalization of Theorem \ref{th:karolyi_abelian}, due to the following lemma.
\begin{lemma}\label{lem:simple_lemma}
If $Z$ is a subset of $A$ such that $Z \cap A^\times \ne
\emptyset$, then $\omega(Z) \ge \mathfrak{p}(\mathbb A)$.
\end{lemma}
\begin{proof}
Pick $z_0 \in Z \cap A^\times$ using that $Z \cap A^\times \ne \emptyset$. If $Z$ is a singleton, the assertion is trivial since then $\inf_{z \in Z \setminus \{z_0\}} \ord(z-z_0) = |\mathbb N|$. In the other case, taking $z \in Z \setminus \{z_0\}$ gives $\ord(z - z_0) \ge
\mathfrak{p}(\mathbb A)$ by the definition of $\mathfrak{p}(\mathbb A)$.
\end{proof}
Lemma \ref{lem:simple_lemma}
applies, on the level of groups, to \textit{any} non-empty subset (see Corollary \ref{prop:groupal_case} below), and
the stated inequality is strict in significant cases: For a concrete example, pick $k,q \in \mathbb N^+$ and set $m := qk$ and $
X := \{(1 + ik) \bmod m: i = 1, \ldots, q\}$. Then observe that $|2X| = \Omega_{\mathbb Z/m\mathbb Z}(X,X) = q$, while
$\mathfrak{p}(\mathbb Z/m \mathbb Z)$ is the
smallest prime, say $p$, dividing $m$, to the effect that $\mathfrak{p}(\mathbb Z/m\mathbb Z)$ is
``much'' smaller than $\Omega_{\mathbb Z/m\mathbb Z}(X,X)$ if $p$ is ``much'' smaller than
$q$.

Theorem \ref{th:strong_main} can be ``symmetrized'' and further strengthened in the
case where each summand generates a commutative subsemigroup, leading to the following corollaries, whose proofs are straightforward in the light of Definition \ref{def:davenport_constant}:
\begin{corollary}\label{cor:1}
Assume $\mathbb A$ is cancellative and let $X,Y$ be subsets of $A$ such that $\langle X \rangle$ is commutative. Then, $|X+Y| \ge
\min(\omega(X),|X|+|Y|-1)$ if both of $X$ and $Y$ are finite and non-empty, and $|X+Y| \ge \Omega(X,Y)$ otherwise.
\end{corollary}
\begin{corollary}\label{cor:2}
If $\mathbb A$ is cancellative and $X,Y$ are subsets of
$A$ such that both of $\langle X
\rangle$ and $\langle Y \rangle$ are commutative, then $|X+Y| \ge \Omega(X,Y)$.
\end{corollary}
Moreover, the result specializes to groups as follows:
\begin{corollary}\label{prop:groupal_case}
If $\mathbb A$ is a group and $X,Y$ are non-empty subsets of
$A$ such that $\langle Y \rangle$ is commutative. Then,
$|X+Y| \ge \min(\omega(Y), |X|+|Y|-1)$, where
\begin{displaymath}
\omega(Y) = \sup_{y_0 \in Y} \inf_{y \in Y \setminus \{y_0\}} \ord(y - y_0),
\end{displaymath}
and indeed $\omega(Y) = \max_{y_0 \in Y} \inf_{y \in Y \setminus \{y_0\}} \ord(y - y_0)$ if $Y$ is finite.
\end{corollary}
\begin{proof}
Immediate by Theorem \ref{th:strong_main}, for on the one hand $\mathbb
A$ being a group implies $Y = Y \cap A^\times$, and on the other, a supremum over a
finite set is a maximum.
\end{proof}
The next corollary is now a \textit{partial} generalization of Theorem \ref{th:CD_for_torsion-free_groups} to cancellative semigroups: its proof is straightforward by Corollary \ref{cor:2} and Lemma \ref{lem:simple_lemma}. Here, we say that $\mathbb A$ is torsion-free if $\mathfrak{p}(\mathbb A)$ is infinite (in fact, this is an abstraction of the analogous definition for groups).
\begin{corollary}
\label{cor:weak_kemperman}
If $\mathbb A$ is cancellative and $X,Y$ are non-empty subsets of $A$ such that every element of $A \setminus \{0\}$ has order $\ge |X| + |Y| - 1$ in $\mathbb A$ (this is especially the case when $\mathbb A$ is torsion-free) and either of $\langle X \rangle$ or $\langle Y \rangle$ is abelian, then $|X+Y| \ge |X| + |Y| - 1$.
\end{corollary}
Theorem \ref{th:strong_main} is proved in Section \ref{sec:main_results}. The
argument is inspired by the transformation proof originally used for Theorem \ref{th:basic_CD}
by Davenport in \cite{Daven35}. This leads us to the definition of what we call a \textit{generalized Davenport transform}.
The author is not aware of an earlier use of the same technique in the literature, all the more in relation to semigroups. With few
exceptions, remarkably including \cite{HRodseth} and A.G. Vosper's original proof of his famous
theorem on critical pairs \cite{Vosper56}, even the ``classical'' Davenport transform has not been greatly considered by practitioners in the area, especially in comparison with similar ``technology'' such as the Dyson transform
\cite[p. 42]{Natha}.
\begin{remark}\label{rem:malcev}
A couple of things are worth mentioning before proceeding. While every \textit{commutative} cancellative semigroup embeds as
a subsemigroup into a group (as it follows from the standard construction of the group of fractions of a commutative monoid; see \cite[Chapter I, Section 2.4]{BourAlgI}), nothing similar is true in the non-commutative case. This is linked to a well-known question in the theory of semigroups, first answered by A.I. Mal'cev in \cite{Malcev37}, and serves as a fundamental motivation
for the present paper, in that
it shows that the study of sumsets in cancellative semigroups cannot be
systematically reduced,
in the absence of commutativity, to the case of groups (at the very least, not in any
obvious way).

On the other hand, it is true that every cancellative semigroup can be embedded into a cancellative monoid (through the unitization process mentioned in the comments preceding the statement of Theorem \ref{th:karolyi_abelian}, in Section \ref{sec:intro}), to the effect that, for the \textit{specific purposes of the manuscript}, we could have assumed in most of our statements that the ``ambient'' is a monoid rather than a semigroup, but we did differently for the assumption is not really necessary. We will see, however, that certain parts take a simpler form when an identity is made available
somehow, as in the case of lemmas in Section \ref{sec:preliminaries} or in the proof of Theorem \ref{th:strong_main}.
\end{remark}
We provide two applications of Theorem \ref{th:strong_main} in Section \ref{sec:corollaries} (others will be investigated in future work): The first is a generalization of Theorem \ref{th:chowla_theorem}, the second is an improvement on a previous result by \O{}.J. R\o{}dseth \cite[Section 6]{Rodseth} relying on Hall's ``marriage theorem''. As for the former (which is stated below), we will use the following specific notation: Given $m \in \mathbb N^+$ and a non-empty $Z \subseteq \mathbb Z/m\mathbb Z$, we let
\begin{equation}
\label{equ:delta}
\delta_Z := \min_{z_0 \in Z} \max_{z \in Z \setminus \{z_0\}}  \gcd(m, z - z_0)
\end{equation}
if $|Z| \ge 2$, and $\delta_Z := 1$ otherwise. Based on this, the next result holds:
\begin{corollary}\label{th:chowla_plus}
For an integer $m \ge 1$ let $X$ and $Y$ be non-empty subsets of $\mathbb Z/m\mathbb Z$ and define $\delta := \min(\delta_X, \delta_Y)$. Then,
$|X+Y| \ge \min(\delta^{-1} m,
|X|+|Y|-1)$.
More in particular, $|X+Y| \ge \min(m, |X|+|Y|-1)$ if there exists $y_0
\in Y$ such that $m$ is prime with $y - y_0$ for
each $y \in Y \setminus \{y_0\}$ (or dually with $X$ in place of $Y$).
\end{corollary}
In fact, Corollary \ref{th:chowla_plus} contains Chowla's theorem (Theorem \ref{th:chowla_theorem}) as a special case: With the same notation as above, it is enough to assume that the identity of $\mathbb Z/m\mathbb Z$ belongs to $Y$, and $\gcd(m,y) = 1$ for each non-zero $y \in Y$. Furthermore, it is clear from \eqref{equ:delta} that the result is a strengthening of Pillai's theorem (Theorem \ref{th:pillai_theorem}).

Many questions
arise. Most notably: Is it possible to further extend Corollary \ref{cor:2}
in such a way to get rid of
the assumption that summands generate commutative subsemigroups?
This sounds particularly significant, for a positive answer would provide a
comprehensive generalization of
about all the extensions of the Cauchy-Davenport theorem reviewed in
Section \ref{sec:intro}, and remarkably of Theorems
\ref{th:karolyi_finite} and \ref{th:CD_for_torsion-free_groups}.
%
\section{Preliminaries}\label{sec:preliminaries}
This section collects basic results used later to
introduce the
generalized Davenport transforms and prove Theorem \ref{th:strong_main}. Some
proofs
are direct and standard (and thus omitted without further explanation),
but we have no reference to anything
similar in the context of semigroups, so we include them
here for completeness.
\begin{lemma}\label{lem:trivialities}
Pick $n \in \mathbb N^+$ and subsets $X_1, Y_1,\ldots, X_n, Y_n$ of $A$ such that $X_i \subseteq Y_i$ for each $i$. Then, $\sum_{i=1}^n X_i
\subseteq \sum_{i=1}^n Y_i$ and $\big|\sum_{i=1}^n X_i\big| \leq \big|\sum_{i=1}^n Y_i\big|$.
\end{lemma}
\begin{lemma}\label{lem:translation_invariance}
Assume $\mathbb A$ is cancellative and pick an integer $n \ge 2$ and non-empty $X_1, \ldots, X_n \subseteq A$. Then,
$\big|\sum_{i = 2}^n X_i\big| \le \big|\sum_{i=1}^n X_i\big|$ and $\big|\sum_{i = 1}^{n-1} X_i\big| \le \big|\sum_{i=1}^n X_i\big|$.
\end{lemma}
For the next lemma, whose proof is straightforward by a routine induction, we assume that $0 \cdot \kappa := \kappa \cdot 0 := 0$ for every cardinal $\kappa$.
\begin{lemma}\label{lem:upper_bounds}
For $n \in \mathbb N^+$ and $X_1, \ldots, X_n \subseteq A$  it holds $\big|\sum_{i=1}^n X_i\big| \le
\prod_{i=1}^n |X_i|$.
\end{lemma}
Let $X,Y \subseteq A$. No matter if $\mathbb A$ is cancellative, nothing similar to Lemmas
\ref{lem:translation_invariance} and \ref{lem:upper_bounds} applies, in general,
to the difference set $X-Y$, in the sense that this can be infinite even if both of $X$ and $Y$ are finite. On another hand, we get by symmetry and Lemma \ref{lem:translation_invariance} that, in the presence of
cancellativity, the cardinality of the sumset $X+Y$ is
preserved under translation, namely $|z+X+Y| = |X+Y+z| = |X+Y|$ for every $z \in A$. This is a point in
common with the case of groups, save for the fact that we cannot profit from it, at least in general, to ``normalize'' either
of $X$ or $Y$ in such a way as to contain some distinguished element of $A$.
\begin{lemma}\label{lem:incremental}
Let $X$ and $Y$ be subsets of $A$. The
following are equivalent:
\begin{enumerate}[label={\rm(\roman{*})}]
\item\label{item:lem:incremental_i} $X + 2Y \subseteq X + Y$.
\item\label{item:lem:incremental_ii} $X + nY \subseteq X + Y$ for all $n \in \mathbb N^+$.
\item\label{item:lem:incremental_iii} $ X + \langle Y \rangle = X+Y$.
\end{enumerate}
\end{lemma}
\begin{proof}
Points \ref{item:lem:incremental_ii} and \ref{item:lem:incremental_iii} are clearly equivalent,
as $X + \langle Y \rangle = \bigcup_{n=1}^\infty (X + nY)$, and
\ref{item:lem:incremental_i} is obviously implied by \ref{item:lem:incremental_ii}. Thus, we are left to prove that \ref{item:lem:incremental_ii} follows
from \ref{item:lem:incremental_i}, which is immediate (by induction) using that, if
$X+nY \subseteq X+Y$ for some $n \in \mathbb N^+$, then $X+(n+1)Y = (X+nY) + Y
\subseteq (X+Y) +
Y = X + 2Y \subseteq X+Y$.
\end{proof}
The above result is as elementary as central in the plan of the paper, for the properties of the generalized Davenport transform used later, in Section \ref{sec:main_results}, in the proof of Theorem \ref{th:strong_main} are strongly dependent on it.

On another hand, the following lemma shows that, in reference to Theorem
\ref{th:strong_main},
there is no loss of generality in assuming that the ambient semigroup is unital, for any semigroup embeds as a subsemigroup into its unitization.
\begin{lemma}\label{lem:embeddings}
Let $ (B, \star)$ be a
semigroup, $\varphi$ an injective function from $A$ to $B$ such that $\varphi(z_1+z_2) = \varphi(z_1) \star \varphi(z_2)$ for all $z_1, z_2 \in A$, and $X_1,\ldots, X_n \subseteq A$ ($n \in \mathbb N^+$). Then, $|X_1 + \cdots + X_n| = |\varphi(X_1) \star \cdots \star \varphi(X_n)|$.
\end{lemma}
We close the section with a few properties of units. Here and later, given $X \subseteq A$ we use ${\rm C}_A(X)$ for the centralizer of $X$ in $\mathbb A$, namely the set of all $z \in A$ such that $z + x = x + z$ for every $x \in X$.
\begin{lemma}\label{lem:inverses}
Let $\mathbb A$ be a monoid, $X$ a subset
of $A$, and $z$ a unit of $\mathbb A$ with inverse $\tilde
z$. Then the following conditions hold:
\begin{enumerate}[label={\rm(\roman{*})}]
\item\label{item:lem:inverses_i} $X - z = X + \tilde z$, $- z + X = \tilde z + X$ and $|-z
+ X| = |X-z| = |X|$.
\item\label{item:lem:inverses_iii} If $z \in {\rm C}_A(X)$ then $\tilde z \in {\rm
C}_A(X)$; in addition to this, $\langle X - z \rangle$ and
$\langle -z+X\rangle$ are commutative if $\langle X
\rangle$ is commutative.
\end{enumerate}
\end{lemma}
\begin{proof}
\ref{item:lem:inverses_i} By  symmetry, it suffices to prove that $X - z = X +
\tilde
z$ and $|X-z| = |X|$. As for the first identity, it holds $w \in
X-z$ if and only if there exists $x \in X$ such that $w + z = x$, which in
turn is equivalent to $x + \tilde z =  (w + z) + \tilde z =
w$, namely $w \in X+\tilde z$. In order to conclude, it is then sufficient to observe that the function $A \to A: \xi \mapsto \xi + \tilde z$ is a bijection.

\ref{item:lem:inverses_iii} Pick $z \in {\rm C}_A(X)$ and $x \in X$. It is then seen that $x + \tilde z = \tilde z + x$ if and only if $x = (x + \tilde z) + z = \tilde z + x + z$, and this is certainly
verified as our standing assumptions imply $\tilde z + x + z = \tilde z + z + x = x$. It follows that $\tilde z \in {\rm
C}_A(X)$.

Suppose now that $\langle X
\rangle$ is a commutative semigroup and let $v, w \in \langle X - z
\rangle$. By point \ref{item:lem:inverses_i} above, there exist $k,\ell \in \mathbb N^+$ and $x_1, \ldots, x_k, y_1, \ldots, y_\ell \in X$ such that $v = \sum_{i=1}^k (x_i + \tilde z)$ and $w = \sum_{i=1}^\ell (y_i + \tilde z)$, to the effect that $v + w = w + v$ by induction on $k + \ell$ and the observation that for all $u_1, u_2 \in X$ it holds
\begin{displaymath}
(u_1 + \tilde z) + (u_2 + \tilde z) = u_1 + u_2 +
2\tilde z = u_2 + u_1 + 2\tilde z = (u_2 +
\tilde z) + (u_1 + \tilde z),
\end{displaymath}
where we use that $\tilde
z \in {\rm C}_A(X)$, as proved before, and $\langle X
\rangle$ is commutative. Hence, $\langle X -
z\rangle$ is commutative too, which completes the proof by symmetry.
\end{proof}
\begin{remark}\label{rem:subtle}
There is a subtleness in Definition \ref{def:davenport_constant} which we have ``overlooked'' so far, but should be noticed.
For, suppose that $\mathbb A$ is a monoid and pick $x,y \in A$. In
principle, $x-y$ and $-y+x$ are not elements of $A$: In fact, they are
(difference) \textit{sets}, and no other meaningful interpretation is possible a priori. However, if $y$ is a unit of $\mathbb A$ and $\tilde y$ is the
inverse of $y$, then $x - y = \{x + \tilde y\}$ and $-y + x =
\{\tilde y + x\}$ by point \ref{item:lem:inverses_i} of Lemma \ref{lem:inverses}, and we are allowed
to identify $x - y$ with $x + \tilde y$ and $-y + x$ with $\tilde y + x$, which will
turn to be useful in various places.
\end{remark}
%
%
\section{The Davenport transform revisited}\label{sec:davenport_transform}
%
%
As mentioned in Section \ref{sec:statements}, Davenport's proof
\cite[Statement A]{Daven35} of Theorem \ref{th:basic_CD} is a transformation
proof. Assuming that $\mathbb A$ is a \textit{commutative group}, the idea is to map a pair $(X,Y)$ of non-empty subsets
of $A$
to a new pair $(X, Y^\prime)$, which is smaller than
$(X,Y)$ in an appropriate sense, and specifically such that
\begin{displaymath}
|Y^\prime| <
|Y|,\quad |X + Y^\prime| + |Y| \le |X + Y| + |Y^\prime|.
\end{displaymath}
We then refer to $(X, Y^\prime)$ as a Davenport
transform of $(X,Y)$; see, for instance, \cite{HRodseth}. For this to be possible, the classical approach
requires that $X+2Y \not\subseteq X+Y$ and $0 \in
Y$, to the effect
that $|Y| \ge 2$.

As expected, many difficulties arise when attempting to adapt the
same approach to semigroups, all the more if these are non-commutative.
Even the possibility of embedding a semigroup into a monoid
does not resolve anything, since the fundamental problem is that,
contrary to
the case of groups, cardinality is not preserved ``under subtraction''.
Namely, if $\mathbb A$ is an arbitrary monoid with identity
$0$ (as intended for the rest of the section, unless differently stated), $X$ is a subset of $A$, and $z$ is an element of $A$, then $|X|$, $|X-z|$
and $|-z+X|$ can be greatly different from each other, even supposing that
$\mathbb A$ is cancellative; cf. point \ref{item:lem:inverses_i} of Lemma \ref{lem:inverses}. Thus,
unless $\mathbb A$ is a group in
disguise or, more generally, embeds as a submonoid into a group, we are not
allowed to assume, for instance, that $0 \in Y$ by picking an arbitrary element $y_0 \in Y$ and replacing $(X,Y)$ with the ``shifted'' pair $(X+y_0, -y_0 + Y)$; cf. the comments following Lemma \ref{lem:upper_bounds}.

In fact, the primary goal of this
section is to show that, in spite of these issues, Davenport's original
ideas can be extended and used for a proof of Theorem
\ref{th:strong_main}.

To start with, let $X$ and $Y$ be subsets
of $A$ such that $mX+2Y \not\subseteq X+Y$ for some positive integer $m$. For the sake of brevity, define
\begin{displaymath}
Z := (mX+2Y) \setminus (X+Y).
\end{displaymath}
Our assumptions give $Z \ne \emptyset$. So fix $z \in Z$, and take $x_z \in
(m-1)X$ and $y_z \in Y$ for which $z \in x_z + X + Y + y_z$, where $0X :=
\{0\}$. Finally, set
\begin{equation}\label{equ:defining_Davenport_transformed_pair}
 \tilde Y_z := \{y \in Y: z \in x_z + X + Y + y\},
\quad Y_z := Y
\setminus \tilde Y_z.
\end{equation}
We refer to
$(X, Y_z)$ as a \textit{generalized} Davenport transform of
$(X,Y)$ (relative to $z$), and based on this notation we have the next
proposition:
\begin{proposition}\label{prop:properties_of_the_modified_Davenport_transform}
If $Y_z \ne \emptyset$, then the triple $(X,Y_z, \tilde Y_z)$ satisfies the
following:
\begin{enumerate}[label={\rm(\roman{*})}]
 \item\label{item:D_transform_i} $Y_z$ and $\tilde Y_z$ are non-empty disjoint proper subsets of $Y$,
and $\tilde Y_z = Y \setminus Y_z$.
 \item\label{item:D_transform_ii} If $\mathbb A$ is cancellative, then $(x_z + X +
Y_z) \cup (z - \tilde Y_z) \subseteq x_z + X + Y$.
\item\label{item:D_transform_iii} $(x_z + X
+
Y_z) \cap (z - \tilde Y_z) = \emptyset$ if $\langle Y \rangle$ is commutative.
\item\label{item:D_transform_iv} If $\mathbb A$ is cancellative, then $|z-\tilde Y_z|
\ge |\tilde Y_z|$.
\item\label{item:D_transform_v}  $|X+Y| + |Y_z| \ge |X+Y_z| + |Y|$ if $\mathbb A$ is cancellative and $\langle
Y \rangle$ commutative.
\end{enumerate}
\end{proposition}
\begin{proof}
\ref{item:D_transform_i} $Y_z$ and $\tilde Y_z$ are non-empty because $\tilde y_z \in
Y_z$ by construction. Also,
(\ref{equ:defining_Davenport_transformed_pair}) gives $Y_z, \tilde Y_z
\subseteq Y$ and $Y_z \cap \tilde Y_z = \emptyset$, so that $Y
\setminus Y_z = Y \setminus (Y \setminus
\tilde Y_z) = \tilde Y_z$ and $Y_z, \tilde Y_z
\subsetneq Y$.

\ref{item:D_transform_ii} Since $Y_z \subseteq Y$ by point \ref{item:D_transform_i} above, $x_z+X+Y_z \subseteq
x_z+X+Y$ by Lemma
\ref{lem:translation_invariance}. On the
other hand,
if $w \in z - \tilde Y_z$ then there exists $y \in \tilde Y_z$ such that
$z = w + y$. But $y \in \tilde Y_z$ implies by
(\ref{equ:defining_Davenport_transformed_pair}) that  $z = \tilde w +
y$ for some $\tilde w \in x_z+X+Y$, whence $w = \tilde w$ by right
cancellativity, namely $w \in x_z+X+Y$.

\ref{item:D_transform_iii} Assume the contrary and let $w \in (x_z+ X +
Y_z) \cap (z - \tilde Y_z)$. There then exist $x \in X$, $y_1 \in Y_z$
and $y_2
\in \tilde Y_z$ such that $w = x_z + x + y_1$ and $z = w + y_2$. Using that
$\langle Y
\rangle$ is commutative, it follows that
$z = x_z + x + y_1 + y_2 = x_z + x + y_2 + y_1$,
which in turn implies $y_1 \in \tilde Y_z$ by
(\ref{equ:defining_Davenport_transformed_pair}), since
$Y_z, \tilde Y_z \subseteq Y$ by point \ref{item:D_transform_i}. This is, however, absurd as
$Y_z \cap \tilde Y_z = \emptyset$, by the same point \ref{item:D_transform_i}.

\ref{item:D_transform_iv} We have from
(\ref{equ:defining_Davenport_transformed_pair}) that for each $y \in \tilde Y_z$
there exists $w \in x_z+X+Y$ such that $z = w + y$, and hence $w
\in z - \tilde Y_z$. On the other hand, since $\mathbb A$ is left
cancellative, it cannot happen that $w + y_1 = w + y_2$ for some $w \in
\mathbb A$ and distinct $y_1, y_2 \in \tilde Y_z$. Thus, $\tilde Y_z$ embeds
as a set into $z - \tilde Y_z$, with the result that $|z -
\tilde Y_z| \ge |\tilde Y_z|$.

\ref{item:D_transform_v} Since $\mathbb A$ is cancellative and $X \ne \emptyset$ (otherwise
$Z=\emptyset$), we have $|X+Y| \ge \max(|X|,|Y|)$ by symmetry and Lemma
\ref{lem:translation_invariance}. This implies the claim if $Y$ is infinite, since then either $|X+Y| > |Y|$, and hence
\begin{displaymath}
|X+Y| + |Y_z| = |X| = |X+Y_z| + |Y|,
\end{displaymath}
or instead $|X+Y| = |Y|$, and accordingly
\begin{displaymath}
|X+Y_z| + |Y_z| = |Y| = |X+Y_z| + |Y|.
\end{displaymath}
We are using here the axiom of choice, which is assumed in the background as part of our foundations, to say that $|X+Y|= \max(|X|,|Y|)$ if $X$ and $Y$ are both infinite.
So we are left with the case when $Y$ is
finite, for which the inclusion-exclusion
principle, points \ref{item:D_transform_ii}-\ref{item:D_transform_iv} and Lemma
\ref{lem:translation_invariance} give, by symmetry, that
\begin{displaymath}
\begin{split}
 |X+Y| & = |x_z+X+Y| \ge |x_z+X+Y_z| + |z-\tilde Y_z| = \\
       & = |X+Y_z| + |z-\tilde Y_z| \ge |X+Y_z| + |\tilde Y_z|.
\end{split}
\end{displaymath}
But $\tilde Y_z = Y\setminus Y_z$ and $Y_z \subseteq Y$ by point \ref{item:D_transform_i} above, so in the end we get $|X+Y| \ge |X+Y_z| + |Y| - |Y_z|$, and the proof is complete.
\end{proof}
\begin{remark}\label{rem:applying_Davenport}
To apply the generalized Davenport transform to
Theorem \ref{th:strong_main}, it will be enough to consider the case where
$m = 1$, for which it is easily seen that
$0 \in Y_z$ if $0 \in Y$ (we continue with the notation
from above), as otherwise $z \in X+Y$, contradicting
the fact that $z \in (X+2Y) \setminus (X+Y)$. However, it seems intriguing that
the same machinery can
be used, at least in principle, even if $m \ge 2$ in so far as there is a way
to prove that $Y_z$ is non-empty.
\end{remark}
\section{The proof of the main theorem}
\label{sec:main_results}
Lemma \ref{prop:properties_of_the_modified_Davenport_transform} is used here
to establish the main contribution of the paper.
\begin{proof}[Proof of Theorem \ref{th:strong_main}]
Since every semigroup embeds as a subsemigroup into its unitization, and the unitization
of a cancellative semigroup is cancellative in its own right,
Lemma \ref{lem:embeddings} and Definition
\ref{def:davenport_constant} imply that there is no loss
of generality in assuming, as we do, that $\mathbb A$ is unital.

Thus, suppose by contradiction that the theorem is
false. There then exists a pair $(X,Y)$ of subsets of $A$ for which  $|X+Y| <
\min(\omega(Y), |X|+|Y|-1)$, whence
\begin{equation}\label{equ:cardinality_not_too_small_not_too_large}
2 \le |X|,|Y| < |\mathbb N|.
\end{equation}
In fact, if either of $X$ or $Y$ is a singleton or infinite then $|X+Y| =
\max(|X|,|Y|)$,
and Definition \ref{def:davenport_constant} gives
$|X+Y| = \Omega(X,Y)$, contradicting the standing assumptions. It
follows from (\ref{equ:davenport_constant}) and
(\ref{equ:cardinality_not_too_small_not_too_large}) that
\begin{equation}\label{eq:minimality}
 |X+Y| < \sup_{y_0 \in Y \cap A^\times} \inf_{y \in Y\setminus
\{y_0\}} \ord(y - y_0),\quad
|X+Y| \le |X| + |Y| - 2.
\end{equation}
Again without loss of generality, we also assume that
$|X|+|Y|$ is
minimal over the pairs of subsets of $A$ for which
(\ref{equ:cardinality_not_too_small_not_too_large}) and (\ref{eq:minimality})
are presumed to hold.

Now, since  $|X+Y|$ is finite, thanks to
(\ref{equ:cardinality_not_too_small_not_too_large}) and Lemma
\ref{lem:upper_bounds}, we get by (\ref{eq:minimality}) and the same equation (\ref{equ:cardinality_not_too_small_not_too_large}) that there exists $\tilde y_0 \in Y \cap A^\times$ such that
\begin{equation}\label{eq:rappresentante}
 |X+Y| < \inf_{y \in Y\setminus \{\tilde y_0\}}
\ord(y - \tilde
y_0) = \min_{y \in Y\setminus \{\tilde y_0\}}
\ord(y - \tilde
y_0).
\end{equation}
So letting $0$ be the identity of $ \mathbb A$ and taking $ W_0 :=
Y - \tilde y_0$ imply
\begin{equation}\label{eq:riassuntivo}
  |X+W_0|
< \min_{w \in W_0\setminus \{0\}} \ord(w),\quad
|X+W_0| \le |X|+|W_0| - 2
\end{equation}
in view of (\ref{eq:minimality}) and
(\ref{eq:rappresentante}). In fact, on the one hand $|Y-\tilde y_0| = |Y|$ and $|X+Y-\tilde y_0| =|X+Y|$ by point
\ref{item:lem:inverses_i} of Lemma \ref{lem:inverses}, and on the other hand, $y \in Y \setminus
\{\tilde y_0\}$ only if $y - \tilde y_0 \in (Y - \tilde
y_0)\setminus \{0\}$, as well as $w \in (Y - \tilde
y_0)\setminus \{0\}$ only if $w + \tilde y_0 \in Y \setminus \{\tilde y_0\}$
(see also Remark \ref{rem:subtle}). We claim that
\begin{equation}\label{equ:failed_inclusion}
Z :=
(X+2W_0) \setminus (X+W_0) \ne \emptyset.
\end{equation}
For, suppose the contrary. Then, $X+W_0 = X+\langle W_0 \rangle$ by
Lemma \ref{lem:incremental}, so that
\begin{displaymath}
 |X+W_0| = |X+\langle W_0 \rangle| \ge |\langle W_0
\rangle| \ge \max_{w \in W_0} \ord(w) \ge
\min_{w \in W_0 \setminus \{0\}} \ord(w),
\end{displaymath}
where we use, in particular, Lemma
\ref{lem:translation_invariance} for the first inequality and the fact that
$|W_0| \ge 2$ for the last one. But this contradicts
(\ref{eq:riassuntivo}), so (\ref{equ:failed_inclusion}) is proved.

Pick $z \in Z$ and let
$(X,W_0^\prime)$ be a generalized Davenport transform of $(X,W_0)$ relative to
$z$.
Since $\langle Y \rangle$ is a commutative subsemigroup of $\mathbb A$ (by hypothesis), the same is true for $\langle W_0 \rangle$, by point \ref{item:lem:inverses_iii} of
Lemma \ref{lem:inverses}. Moreover, $0 \in W_0$, and thus
\begin{equation}\label{eq:W0_non_vuoto}
0 \in W_0^\prime \ne \emptyset, \quad W_0^\prime \subsetneq W_0,
\end{equation}
when taking into account
Remark \ref{rem:applying_Davenport} and point \ref{item:D_transform_i} of Proposition
\ref{prop:properties_of_the_modified_Davenport_transform}. As a consequence,
point \ref{item:D_transform_v} of the same Proposition
\ref{prop:properties_of_the_modified_Davenport_transform} yields, together with
(\ref{eq:riassuntivo}),
that
\begin{displaymath}
|X+W_0^\prime| + |W_0| \le |X+W_0| + |W_0^\prime| \le |X|+|W_0| - 2 +
|W_0^\prime|,
\end{displaymath}
which means, since $|W_0| = |Y - \tilde y_0| = |Y| < |\mathbb N|$ by
(\ref{equ:cardinality_not_too_small_not_too_large}) and the above, that
\begin{equation}\label{eq:bound_su_X_e_Y0trasformato}
|X+W_0^\prime| \le |X|+|W_0^\prime| - 2.
\end{equation}
It follows from (\ref{eq:W0_non_vuoto}) that $1 \le |W_0^\prime| < |W_0|$, and
in fact $|W_0^\prime| \ge 2$, as otherwise we would have $
|X| = |X+W_0^\prime| \le |X| - 1$ by (\ref{eq:bound_su_X_e_Y0trasformato}), in contradiction with the fact that $|X| < |\mathbb N|$ by
(\ref{equ:cardinality_not_too_small_not_too_large}). To summarize, we have found that
\begin{equation}\label{eq:summary}
2 \le |W_0^\prime| < |W_0| < |\mathbb N|.
\end{equation}
Furthermore,
(\ref{eq:riassuntivo}) and (\ref{eq:W0_non_vuoto}) entail
that
\begin{equation}\label{eq:towards_a_contradiction}
 |V_0+W_0^\prime| \le |V_0+W_0| < \min_{w \in W_0^\prime
\setminus\{0\}} \ord(w),
\end{equation}
where we use that $\min(C_1) \le \min(C_2)$ if $C_1$ and
$C_2$ are sets of cardinal numbers with $C_2 \subseteq C_1$. Thus, since $0 \in
{W_0^\prime} \cap A^\times$, we get by (\ref{eq:towards_a_contradiction}) that
\begin{equation}
 |X+W_0^\prime| < \sup_{w_0 \in {W_0^\prime} \cap A^\times} \min_{w \in
W_0^\prime \setminus\{w_0\}} \ord(w),
\end{equation}
which is however in contradiction, due to
(\ref{equ:cardinality_not_too_small_not_too_large}),
(\ref{eq:bound_su_X_e_Y0trasformato}) and
(\ref{eq:summary}), with
the minimality of $|X|+|Y|$, for $|W_0^\prime| < |W_0| = |Y|$,
and hence $|X|+|W_0^\prime| < |X|+|Y|$.
\end{proof}
\section{A couple of applications}\label{sec:corollaries}
First, we show how to use Theorem \ref{th:strong_main} to prove the extension of Chowla's theorem for composite moduli mentioned in  Section \ref{sec:statements}.

\begin{proof}[Proof of Corollary \ref{th:chowla_plus}]
The claim is trivial if either of $X$ or $Y$ is a singleton. Otherwise, $\mathbb Z/m\mathbb Z$ being a commutative finite group and $\ord(z-z_0) = m/\gcd(m, z-z_0)$ for $z,z_0 \in \mathbb Z/m\mathbb Z$ imply $|X+Y| \ge \min(\omega(Y),|X|+|Y|-1)$ by Corollary \ref{prop:groupal_case}, where
\begin{displaymath}
\omega(Y) = \max_{y_0 \in Y} \min_{y \in Y \setminus \{y_0\}} \ord(y-y_0) = m \cdot \max_{y_0 \in Y} \min_{y \in Y \setminus \{y_0\}} \frac{1}{\gcd(m,y-y_0)} = \delta_Y^{-1} m.
\end{displaymath}
Now in an entirely similar way, it is found, in view of Corollary \ref{cor:1}, that
\begin{displaymath}
|X+Y| \ge \min(\delta_X^{-1} m, |X|+|Y|-1).
\end{displaymath}
This concludes the proof, considering that $\delta_Y = 1$ if there exists $y_0 \in Y$ such that $m$ is coprime with $y - y_0$ for every $y \in Y \setminus \{y_0\}$ (and symmetrically with $X$).
\end{proof}
We now use P. Hall's  theorem about distinct representatives
\cite{Hall35}
to
say something on how to ``localize'' some elements of a
sumset.
\begin{theorem}[Hall's theorem]
Let $S_1, \ldots, S_k$ be sets ($k \in \mathbb N^+$).
There then exist (pairwise) distinct
elements $s_1, \ldots, s_k$ such that $s_i \in S_i$ if and only if for each $h = 1, \ldots, k$ the union of any $h$ of $S_1, \ldots, S_k$ contains at least
$h$ elements.
\end{theorem}
More precisely, suppose
$\mathbb A$ is a cancellative
semigroup and let $X,Y$ be non-empty finite subsets of $A$  such that $|X+Y| < \omega(Y)$.
Clearly, this implies $Y \cap A^\times \ne \emptyset$. Define $k := |X|$
and $\ell := |Y|$, and let $x_1, \ldots, x_k$
be a numbering of $X$ and $y_1, \ldots, y_\ell$ a numbering of $Y$. Then
consider the $k$-by-$\ell$ matrix, say $\alpha(X,Y)$, whose entry in the $i$-th row and $j$-th column is $x_i + y_j$.
Any element of $X+Y$ appears in $\alpha(X,Y)$, and
viceversa any entry of $\alpha(X,Y)$ is an element of $X+Y$. Also,
Theorem \ref{th:strong_main} and our hypotheses give $|X+Y| \ge k + \ell - 1$.
So it is natural to try to gain some information about
where in the matrix $\alpha(X,Y)$ to look for $k+\ell-1$ distinct elements of $X+Y$. In this respect we have the following proposition, whose proof is quite similar to the one of a weaker result in \cite[Section 6]{Rodseth}, which is, in turn, focused on the less general case of a group of prime order:
\begin{proposition}
 Assume that $\langle Y
\rangle$ is commutative and let $Z$ be any subset of
$X+Y$ of size $\ell-1$, for instance $Z = x_1 + \{y_1, \ldots, y_{
\ell-1}\}$. Then we can choose one element from each row of $\alpha(X,Y)$ in such a way that $Z$ and these elements form a subset of $X+Y$ of size $k+\ell-1$.
\end{proposition}
\begin{proof}
For each $i = 1, \ldots, k $ let $Z_i := (x_i +
Y)\setminus Z$ and note
that $Z_i$ is a subset of the $i$-th row of
$\alpha(X,Y)$. Then $
Z_{i_1} \cup
 \cdots \cup Z_{i_h} = (\{x_{i_1},
\ldots, x_{i_h}\} + Y) \setminus Z$ for any positive integer $h \le k$ and all distinct $i_1, \ldots, i_h \in \{1, \ldots, k\}$,
with the result that
\begin{displaymath}
 |Z_{i_1} \cup \cdots \cup Z_{i_h}| \ge |\{x_{i_1}, \ldots,
x_{i_h}\} + Y| - |Z| \ge h + \ell - 1 - (\ell - 1) = h,
\end{displaymath}
thanks to Theorem \ref{th:strong_main} and the fact that $|\{x_{i_1}, \ldots, x_{i_h}\} + Y| \le |X+Y| < \omega(Y)$ by Lemma \ref{lem:trivialities} and the assumption that $|X+Y| < \omega(Y)$. It follows from Hall's theorem that we
can find $k$ distinct elements $z_1, \ldots, z_k$ such that $z_1 \in Z_1, \ldots, z_k \in Z_k$, and these, together with the $\ell-1$ elements of $Z$, provide a total amount of $k+\ell-1$ elements in $X+Y$, since $Z \cap Z_1 = \cdots = Z \cap Z_k = \emptyset$ (by construction).
\end{proof}
\section*{Acknowledgements}
The author thanks Andrea Gagna (Universit\`a di Milano, Italy), Carlo Sanna (Universit\`a di Torino, Italy), Alain
Plagne (\'Ecole polytechnique, France), and an anonymous referee for comments that helped to improve this paper significantly.


\begin{thebibliography}{99}
\vskip.3cm
\small
\bibitem[B1]{BourAlgI} {\sc Bourbaki,~N.:} \textsl{Alg\`ebre, Chapitres 1 \`a 3}, \'El\'ements de math\'ematique II, Springer-Verlag, Berlin, 2006 (2nd revised ed.).
\bibitem[B2]{BourSetTh} {\sc Bourbaki,~N.:} \textsl{Th\'eorie des ensembles}, \'El\'ements de math\'ematique I, Springer-Verlag, Berlin, 2006 (reprint ed.).
\bibitem[C]{Cauchy1813} {\sc Cauchy,~A.-L.:} \textit{Recherches sur les nombres}, J. \'Ecole
Polytech. \textbf{9} (1813), 99--116 (reproduced in \textit{Oeuvres}, S\'erie 2, Tome 1, 39--63).
\bibitem[CHS]{Cill10} {\sc Cilleruelo,~A.L. -- Hamidoune,~Y.O. -- Serra,~O.:}
\textit{Addition theorems in acyclic semigroups}, 99--104 in `Additive number theory', Springer, 2010.
\bibitem[Ch]{Chowla35} {\sc Chowla,~I.:} \textit{A theorem on the addition of residue
classes: Application to the number $\Gamma(k)$ in Waring's problems}, Proc.
Indian Acad. Sc. (A), \textbf{2} (1935), 242--243.
\bibitem[D1]{Daven35} {\sc Davenport,~H.:} \textit{On the addition of residue classes}, J.
Lond. Math. Soc. \textbf{10} (1935), 30--32.
\bibitem[D2]{Daven47} {\sc Davenport,~H.:} \textit{A historical note}, J. Lond.
Math. Soc. \textbf{22} (1947), 100--101.
\bibitem[G]{Gerold09} {\sc Geroldinger,~A.:} \textit{Additive Group Theory and
Non-unique Factorizations}, 1--86 in `Combinatorial Number Theory and Additive
Group Theory', Springer, 2009.
\bibitem[H]{Hall35} {\sc Hall,~P.:} \textit{On representatives of subsets}, J. Lond. Math.
Soc. \textbf{10} (1935), 26--30.
\bibitem[HR]{HRodseth} {\sc Hamidoune,~Y.O. -- \sc R\o{}dseth,~\O{}.J.:} \textit{An inverse theorem mod
$p$}, Acta Arith. \textbf{92} (2000), 251--262.
\bibitem[Ho]{Howie96} {\sc Howie,~J.M.:}
\textsl{Fundamentals of semigroup theory}, Clarendon Press, 1995.
\bibitem[K]{Karo05.2} {\sc K\'arolyi,~G.:} \textit{The Cauchy-Davenport theorem in group
extensions}, L'En\-sei\-gne\-ment Math\'ematique \textbf{51} (2005), 239--254.
\bibitem[Ke]{Kemp56} {\sc Kemperman,~J.H.B.:} \textit{On complexes in a semigroup}, Indag.
Math. \textbf{18} (1956), 247--254.
\bibitem[M]{Malcev37} {\sc Mal'cev,~A.I.:} \textit{On the immersion of an algebraic ring
into a field}, Math. Annalen (1)\textbf{113} (1937), 686--691.
\bibitem[N]{Natha} {\sc Nathanson,~M.B.:} \textsl{Additive Number Theory. Inverse Problems
and the Geometry of Sumsets}, GTM \textbf{165}, Springer, 1996.
\bibitem[P]{Pillai} \textsc{Pillai,~S.S.:} \textit{Generalization of a theorem of Davenport on the addition of residue classes}, Proc. Indian Acad. Sc. (A), (3)\textbf{6} (1937), 179--180.
\bibitem[R]{Rodseth} {\sc R\o{}dseth,~\O{}.J.:} \textit{Sumsets mod $p$}, Skr. K. Nor.
Vidensk. Selsk. (Trans. R. Norw. Soc. Sci. Lett.), \textbf{4} (2006), 1--10.
\bibitem[Ru]{Ruzsa09} {\sc Ruzsa,~I.Z.:} \textit{Sumsets and structure}, 87--210 in
`Combinatorial Number Theory and Additive Group
Theory', Springer, 2009.
\bibitem[V]{Vosper56} {\sc Vosper,~A.G.:} \textit{The critical pairs of subsets of a group of prime order}, J. Lond. Math.
Soc. \textbf{31} (1956), 200--205.
\end{thebibliography}
\end{document}